\documentclass[12pt]{amsart}
\usepackage{xy}
\xyoption{all}
\usepackage[pdftitle={Weil representation associated to a finite Heisenberg group}, pdfauthor={Amritanshu Prasad}, pdfkeywords={Weil representation, finite Heisenberg groups}, pdfdisplaydoctitle={true}]{hyperref}
\DeclareMathSizes{12}{10}{8}{6}

\numberwithin{equation}{section}
\theoremstyle{plain}
\newtheorem{theorem}{Theorem}[section]
\newtheorem*{theorem*}{Theorem}
\newtheorem{prop}[theorem]{Proposition}
\newtheorem*{prop*}{Proposition}
\newtheorem{lemma}[theorem]{Lemma}
\newtheorem*{lemma*}{Lemma}
\newtheorem{cor}[theorem]{Corollary}
\newtheorem*{cor*}{Corollary}

\theoremstyle{definition}

\newtheorem*{defn*}{Definition}
\newtheorem{remark}[theorem]{Remark}
\newtheorem*{remark*}{Remark}
\newtheorem{example}[theorem]{Example}

\newcommand{\C}{\mathbf C}
\newcommand{\dash}{\nobreakdash-}

\newcommand{\Hilb}{\mathcal H}
\newcommand{\inv}{^{-1}}

\newcommand{\so}{\sigma_1}
\newcommand{\st}{\sigma_2}
\newcommand{\adjoint}[1]{{#1}^*}

\newcommand{\W}{\mathcal W}
\newcommand{\Z}{\mathbf Z}

\DeclareMathOperator{\Aut}{Aut}
\DeclareMathOperator{\End}{End}
\DeclareMathOperator{\tr}{tr}

\title[The Weil representation]{On character values and decomposition\\of the Weil representation associated\\to a finite abelian group}
\author{{Amritanshu Prasad}}

\address{\href{http://www.imsc.res.in/\string~amri}{The Institute of Mathematical Sciences, Chennai.}}
\email{amri@imsc.res.in}
\begin{document}
\begin{abstract}
We develop a simple algebraic approach to the study of the Weil representation associated to a finite abelian group. 
As a result, we obtain a simple proof of a generalisation of a well-known formula for the absolute value of its character.
 We also obtain a new result about its decomposition into irreducible representations.
 As an example, the decomposition of the Weil representation of $Sp_{2g}(\Z/N\Z)$ is described for odd $N$.
\end{abstract}
\maketitle
\section{Introduction}
Hermann Weyl introduced the Heisenberg group as part of his mathematical formulation of quantum kinematics \cite{Weyl50}.
The fundamental property of Heisenberg groups, now known as the Stone-von Neumann theorem, was predicted by Weyl, proved in the real case by Marshall Stone \cite{Stone30} and John von Neumann \cite{vN31}, and extended by George Mackey to locally compact abelian groups \cite{Mackey49}.
A consequence of the Stone-von Neumann-Mackey theorem is that between any two irreducible unitary representations of the Heisenberg group with identity central character, there is a unitary intertwiner which is unique up to scaling.
An automorphism of the Heisenberg group that fixes its centre can be used to twist an irreducible representation of the Heisenberg group with identity central character to give another one.
The resulting intertwiners were studied by Andr\'e Weil with a view towards applications in number theory \cite{MR0165033}.
They give rise to a projective representation of the group of automorphisms of the Heisenberg group that fix its centre, known as the Weil representation.

The Heisenberg group provides an elegant conceptual framework through which one may view the theory of the Fourier transform (see Example~\ref{example:FT}).
In this framework, the operators in the Weil representation and the Fourier transform have the same origin.
The role of the Heisenberg group in harmonic analysis is discussed in Roger Howe's expository article \cite{MR578375}.

Our interest in the Weil representation comes from representation theory and number theory.
The Weil representation can be used to construct all the representations of $SL(2)$ over a finite field \cite{MR0219635}, and plays an important role in the construction of strongly cuspidal representations of groups over local rings \cite{MR0233931,MR0396859,MR0492087,AOPS}, which are in turn used to construct supercuspidal representations of groups over non-Archimedean local fields.

Roger Howe gave a simple formula for the absolute value of the character of the Weil representation when the underlying abelian group is the additive group of a finite dimensional vector space over a field of odd order \cite[Proposition~2]{MR0316633}.
Heisenberg groups associated to more general finite abelian groups are also of great interest, and arise in, for example, the study of abelian varieties \cite{MR0282985} and representations of automorphism groups of torsion modules over local rings \cite{MR2456275}.
Here we give a different proof of Howe's result, which besides being simpler, works when the underlying abelian group is any abelian group of odd order.
In fact, we deduce it from a more general result, Theorem~\ref{theorem:abs-trace}, which holds even when the abelian group has even order.
Howe's surprising proof proceeds by first establishing Corollary~\ref{cor:abs-char} and deducing Theorem~\ref{theorem:absolute-char-odd} as a formal consequence.
We deduce both Theorem~\ref{theorem:absolute-char-odd} and Corollary~\ref{cor:abs-char} from the more general Theorem~\ref{theorem:abs-trace}, which can not be proved by Howe's method.

Our method also proves to be useful in understanding the decomposition of the Weil representation into irreducible summands.
The main result (Theorem~\ref{theorem:decomposition}) gives a basis for the space of self-intertwiners of the Weil representation of any group of automorphisms of the Heisenberg group that fix its centre.
Again, the results take a simpler form when the underlying abelian group has odd order.
In Section~\ref{sec:example} we illustrate our results by working out the decomposition of the Weil representation of $Sp_{2g}(\Z/N\Z)$ for odd $N$.

Lemma~\ref{lemma:basis} and the fact that it can be used to compute the absolute value of the the character appears in many places in the literature (see, for example, \cite[Remark~2.11]{MR1224052}).

We are grateful to K.~R.~Parthasarathy, Dipendra Prasad, S.~Ramanan and M.~K.~Vemuri for some stimulating discussions which led us towards the ideas in this paper.

\section{The Heisenberg Group}
\label{sec:Heisenberg-group}
Let $A$ be a finite abelian group.
Let $\hat A$ denote its Pontryagin dual, the group of homomorphisms $A\to U(1)$.
On $L^2(A)$ we have
\begin{itemize}
\item Translation operators: $T_xf(u)=f(u-x)$ for $x\in A$.
\item Modulation operators: $M_\chi f(u)=\chi(u)f(u)$ for $\chi\in \hat A$.
\end{itemize}
Let $K=A\times \hat A$.
For $k=(x,\chi)\in K$, let $W_k=T_xM_\chi$.
The unitary operators $W_k$ on $L^2(A)$ are called Weyl operators.
Operators of the form $zW_k$, with $z\in U(1)$, $k\in K$ form a subgroup $G$ of $U(L^2(A))$ under composition, which is known as the Heisenberg group associated to $A$.
When $k=(x,\chi)$ and $l=(y,\lambda)$ the group law is described by
\begin{equation}
  \label{eq:5}
  W_k W_l=\chi(y)W_{k+l}.
\end{equation}
If $\pi(zW_k)=k$ and $j(z)=zW_0$,
then $\pi:G\to K$ and $j:U(1)\to G$ are homomorphisms.
The image of $j$, which is the kernel of $\pi$, is the centre of $G$.
\section{The theorem of Stone, von Neumann and Mackey}
\label{sec:SvNM}
\begin{theorem}
  \label{theorem:SvNM}
  No non-trivial proper subspace of $L^2(A)$ is invariant under all the operators in $G$.
  Moreover, if $\rho:G\to U(\Hilb)$ is an irreducible unitary representation such that $\rho(j(z))=zI$, there exists an isometry $W:L^2(A)\to \Hilb$ such that $W(gf)=\rho(g)W(f)$ for all $f\in L^2(A)$ and $g\in G$.
\end{theorem}
Though this is an easy case of a well-known theorem, we include a proof since the ideas in this proof will be used later. 
We begin with
\begin{lemma}
  \label{lemma:basis}
  $\{W_k\;:\;k\in K\}$ is a basis of $\End_\C(L^2(A))$.
\end{lemma}
\begin{proof}
  Define a representation of $K$ on $\End_\C(L^2(A))$ by
  \begin{equation*}
    \kappa(k)(\phi)=W_k\phi\adjoint{W_k}.
  \end{equation*}
  For $k=(x,\chi)$ and $l=(y,\lambda)$ in $K$, one easily verifies that
  \begin{equation*}
    \kappa(k)(W_l)=\chi(y)\lambda(x)^{-1}W_l
  \end{equation*}
  so that $W_l$ is an eigenvector for $\kappa$ with eigencharacter
  \begin{equation*}
    X_{l}(k)=\chi(y)\lambda(x)^{-1}.
  \end{equation*}
  The map $l\mapsto X_l$ is an isomorphism of $K$ onto its Pontryagin dual.
  Thus, the $W_l$'s, being eigenvectors with distinct eigencharacters, are $|K|$ linearly independent elements of $\End_\C(L^2(A))$, hence a basis.
\end{proof}
\begin{remark}
  This proof of Lemma~\ref{lemma:basis} also gives a decomposition of the representation $\kappa$ of $K$ on $\End_\C(L^2(A))$, see also \cite[Lemma~3.1]{iafhg}.
\end{remark}
\begin{proof}[Proof of Theorem~\ref{theorem:SvNM}]
  By Lemma~\ref{lemma:basis}, any subspace of $L^2(A)$ which is invariant under all the operators $W_k$ must be invariant under all the operators in $\End_\C(L^2(A))$.
  The only such spaces are the trivial space and $L^2(A)$, proving the first assertion.

  By Lemma~\ref{lemma:basis}, any $\phi\in \End_\C(L^2(A))$ has an expansion
  \begin{equation*}
    \phi=\sum_{k\in K} \phi_k W_k
  \end{equation*}
  for uniquely determined scalars $\phi_k$.
  Setting
  \begin{equation*}
    \tilde \rho(\phi)=\sum_{k\in K} \phi_k \rho(W_k)
  \end{equation*}
  makes $\Hilb$ an $\End_\C(L^2(A))$\dash module.
  The irreducibility of $\rho$ implies that $\Hilb$ is a simple $\End_\C(L^2(A))$\dash module.
  The unitarity of $\rho$ implies that the $\C^*$\dash algebra structure on $\End_\C(L^2(A))$ is compatible with the inner product on $\Hilb$.
  Therefore, there exists an isometry $W:L^2(A)\to \Hilb$ such that $W(\phi (f))=\tilde\rho(\phi)(W(f))$ for all $\phi \in \End_\C(L^2(A))$ and all $f\in L^2(A)$.
  Restricting to the Heisenberg group gives $W(gf)=\rho(g)W(f)$.
\end{proof}
\begin{example}
  [The Fourier transform]
  \label{example:FT}
  A representation $\rho$ of $G$ on $L^2(\hat A)$ such that $\rho(j(z))=zI$ is defined by
  \begin{equation*}
    \rho(T_xM_\chi)=M_{-x}T_\chi \text{ for all } x\in A,\;\chi \in \hat A.
  \end{equation*}
  The isometry $W$ of Theorem~\ref{theorem:SvNM} is the Fourier transform $L^2(A)\to L^2(\hat A)$.
\end{example}
\section{Automorphisms of $G$ that fix the centre}
\label{sec:autos}
Suppose that $\sigma:G\to G$ is an automorphism such that $\sigma(j(z))=j(z)$ for each $z\in U(1)$.
Then $\sigma(zW_k)=z\sigma(W_k)$.
Write $\sigma(W_k)$ as $\so(k)W_{\st(k)}$, where $\so:K\to U(1)$ and $\st:K\to K$ are some functions.
Suppose $k=(x,\chi)$ and $l=(y,\lambda)$.
Let $c(k,l)=\chi(y)$.
The condition that $\sigma$ is a group homomorphism puts the following constraints on $\so$ and $\st$:
\begin{gather}
  \label{eq:1}
  \st(k+l)=\st(k)+\st(l),\\
  \label{eq:2}
  \so(k)\so(l)\so(k+l)\inv=c(k,l)c(\st(k),\st(l))\inv
\end{gather}
The identity (\ref{eq:1}) implies that $\st$ is an automorphism of $K$.
The left side of (\ref{eq:2}) is symmetric.
The resulting symmetry of the right side implies that
\begin{equation*}
  e(k,l)=e(\alpha(k),\alpha(l)),
\end{equation*}
where $e(k,l)=c(k,l)c(l,k)\inv$.
Let $B_0(G)$ denote the group of automorphisms of $G$ which fix its centre.
Denoting by $\Aut(K,e)$ the group of automorphisms of $K$ which preserve $e$, we see that $\sigma\mapsto \st$ is a homomorphism $B_0(G)\to \Aut(K,e)$ with kernel $\hat K$.
In particular, $B_0(G)$ is a finite group.
\section{The Weil representation}
\label{sec:Weil-rep}
For each $\sigma\in B_0(G)$, $g\mapsto \sigma(g)$ is again a unitary representation of $G$ on $L^2(A)$ such that $\rho^\sigma(j(z))=zI$.
Therefore, by Theorem~\ref{theorem:SvNM}, there exists a unitary operator $W(\sigma)$ on $L^2(A)$ such that
\begin{equation}
  \label{eq:3}
  W(\sigma)g = \sigma(g)W(\sigma).
\end{equation}
This operator is unique up to multiplication by a unitary scalar.
\section{Absolute value of the character}
\label{sec:absolute}
An ordinary representation $\kappa$ of $B_0(G)$ on $\End_\C(L^2(A))$ is obtained by
\begin{equation*}
  \kappa(\sigma)(\phi)=W(\sigma)\phi\adjoint{W(\sigma)} \text{ for each } \sigma\in B_0(G).
\end{equation*}
The intertwining property (\ref{eq:3}) of $W(\sigma)$ implies that
\begin{equation}
  \label{eq:8}
  \kappa(\sigma)(W_k)=\sigma(W_k)=\so(k)W_{\st(k)} \text{ for each } k\in K.
\end{equation}
Using the basis from Lemma~\ref{lemma:basis} to compute trace gives
\begin{prop}
  \label{prop:trace-kappa}
  For every $\sigma\in B_0(G)$,
  \begin{equation*}
    \tr(\kappa(\sigma))=\sum_{\st(k)=k}\so(k).  
  \end{equation*}
\end{prop}
For any operator $T$ on a finite dimensional Hilbert space $\Hilb$, the trace of the operator $\phi\mapsto T\phi \adjoint T$ on $\End_\C\Hilb$ is $|\tr T|^2$.
Thus an immediate corollary of Proposition~\ref{prop:trace-kappa} is
\begin{theorem}
  \label{theorem:abs-trace}
  For every $\sigma\in B_0(G)$,
  \begin{equation*}
    |\tr W(\sigma)|^2 = \sum_{\st(k)=k} \so(k).
  \end{equation*}
\end{theorem}
\section{Decomposition of the Weil representation}
\label{sec:decomposition}
The uniqueness of $W(\sigma)$ up to scaling implies that, for all $\sigma,\tau\in B_0(G)$, there exists $c(\sigma,\tau)\in U(1)$ such that
\begin{equation}
  \label{eq:4}
  W(\sigma)W(\tau)=c(\sigma,\tau)W(\sigma\tau),
\end{equation}
i.e., $W$ is a projective representation.
For any subgroup $S$ of $B_0(G)$ the twisted group ring $\C[S]_c$ has the same underlying vector space as $\C[S]$, which is spanned by elements $1_s$, $s\in S$, but has multiplication law
\begin{equation*}
  1_\sigma 1_\tau=c(\sigma,\tau)1_{\sigma\tau}.
\end{equation*}
The projective representations of $S$ with cocycle $c$, which satisfy (\ref{eq:4}) are identified with $\C[S]_c$\dash modules in the usual way; $1_\sigma$ inherits the action of $\sigma$.
$\C[S]_c$ is semisimple, and therefore, $L^2(A)$ has a decomposition
\begin{equation*}
  L^2(A)=\sum_\pi m_\pi V_\pi
\end{equation*}
where $\pi$ ranges over the isomorphism classes of simple $\C[S]_c$\dash modules.
By Schur's lemma,
\begin{equation*}
  \End_{\C[S]_c} L^2(A) = \bigoplus_\pi M_{m_\pi}(\C).
\end{equation*}
Therefore, $\sum m_\pi^2=\dim \End_{\C[S]_c}(L^2(A))$.
$\End_{\C[S]_c}(L^2(A))$ consists of those elements $\phi \in \End_\C(L^2(A))$ which commute with $W(\sigma)$ for each $\sigma\in S$.
Suppose $\phi=\sum_{k\in K} \phi_k W_k$.
The condition that $\phi$ commutes with $W(\sigma)$ becomes
\begin{equation*}
  \phi_{\st(k)}=\so(k)\phi_k \text{ for all } k\in K.
\end{equation*}
The group $S$ acts on $K$ by $\sigma\cdot k=\st(k)$ for each $\sigma\in S$ and $k\in K$.
The above equation says that $\phi_k$ determines $\phi_l$ for all $l$ in the $S$\dash orbit of $k$.
Moreover, if there exists $\sigma\in S$ such that $\sigma\cdot k=k$ but $\so(k)\neq 1$ then $\phi_k=0$.
We have proved the following two theorems:
\begin{theorem}
  \label{theorem:centre}
  Let $S$ be a subgroup of $B_0(G)$.
  Then $\End_{\C[S]_c}(L^2(A))$ is spanned by operators of the form $\phi=\sum_{k\in K}\phi_k W_k$ where 
  \begin{equation*}
    \phi_{\st(k)}=\so(k)\phi_k \text{ for all } k\in K,\;\sigma\in S.
  \end{equation*}
\end{theorem}
\begin{theorem}
  \label{theorem:decomposition}
  Let $S$ be any subgroup of $B_0(G)$.
  Suppose that, as a $\C[S]_c$\dash module $L^2(A)$ has a decomposition
  \begin{equation*}
    L^2(A) = \sum_\pi m_\pi V_\pi.
  \end{equation*}
  Then $\sum m_\pi^2$ is the number of $S$\dash orbits in $K$ for which, whenever $\sigma\cdot k=k$, $\so(k)=1$.
\end{theorem}
\section{Simplifications when $|A|$ is odd}
\label{sec:odd}
Throughout this section we assume that $|A|$ is odd.
When $|A|$ is odd, multiplication by $2$ is an automorphism of $|A|$.
We denote its inverse by $x\mapsto x/2$.
For $k=(x,\chi)$ define normalised Weyl operators
\begin{equation*}
  \W_k=\chi(x/2)W_k.
\end{equation*}
From (\ref{eq:5}), we deduce
\begin{equation*}
  \W_k\W_l=\chi(y/2)\lambda(x/2)\inv \W_{k+l}.
\end{equation*}
Let $\tilde c(k,l)=\chi(y/2)\lambda(x/2)\inv$.
Note that $\tilde c(2k,2l)=e(k,l)^2$ for all $k,l\in K$.
Therefore, for any automorphism $\alpha$ of $K$, $\tilde c(\alpha(k),\alpha(l))=\tilde c(k,l)$ if and only if $e(\alpha(k),\alpha(l))=e(k,l)$.

Working as in Section~\ref{sec:autos}, we see that every $\sigma\in B_0(G)$ has the form
\begin{equation*}
  \sigma(\W_k)=\tilde\so(k)\W_{\st(k)}, \text{ where } \tilde\so\in \hat K\text{ and } \st\in \Aut(K,e).
\end{equation*}
In particular, $\sigma(\alpha):\W_k\mapsto \W_{\alpha(k)}$ is in $B_0(G)$ for every $\alpha\in \Aut(K,e)$.
In what follows, we will think of $\Aut(K,e)$ as a subgroup of $B_0(G)$ via the section $\alpha\mapsto \sigma(\alpha)$.
In terms of $\W_k$, (\ref{eq:8}) becomes
\begin{equation*}
  \kappa(\sigma(\alpha))\W_k = \W_{\alpha(k)} \text{ for all } \alpha\in \Aut(K,e),\;k\in K.
\end{equation*}
Therefore Theorem~\ref{theorem:abs-trace} takes the form
\begin{theorem}
  \label{theorem:absolute-char-odd}
  For $\alpha\in \Aut(K,e)$,
  \begin{equation*}
    |\tr W(\sigma(\alpha))|^2=|K^\alpha|.
  \end{equation*}
\end{theorem}
Here $K^\alpha$ is the subgroup of $K$ consisting of elements that are fixed by $\alpha$.
Theorem~\ref{theorem:centre} simplifies to
\begin{theorem}
  \label{theorem:odd-centre}
  Let $S$ be any subgroup of $\Aut(K,e)$.
  Then $\End_{\C[S]_c}(L^2(A))$ consists of those operators $\phi=\sum_{k\in K} \phi_k \W_k$ for which the functions $k\mapsto \phi_k$ are constant on the $S$\dash orbits in $K$.
\end{theorem}
The delta function at the origin and the constant functions are obvious examples of functions that are constant on the $S$\dash orbits for any $S\subset \Aut(K,e)$.
They allow us to find two invariant subspaces.
\begin{lemma}
  \label{lemma:even-odd}
  Let $S$ be any subgroup of $\Aut(K,e)$.
  Let $\delta_0=\W_0$ and $\delta_K=\sum_{k\in K}\W_k$.
  Define
  \begin{equation*}
    \epsilon_\pm=(\delta_0\pm \delta_K/\sqrt{|K|})/2.
  \end{equation*}
  Then $\epsilon_\pm$ are orthogonal central idempotents in $\End_{\C[S]_c}(L^2(A))$ such that $\epsilon_++\epsilon_-=1$.
  As projection operators on $L^2(A)$, $\epsilon_+$ and $\epsilon_-$ are the orthogonal projections onto the subspaces of $L^2(A)$ consisting of the even functions and the odd functions respectively.
\end{lemma}
\begin{proof}
  Clearly, $\delta_0$ is the identity operator on $L^2(A)$.
  For any $f\in L^2(A)$,
  \begin{eqnarray*}
    (\delta_K f)(u) & = & \sum_{l\in K} (\W_l f)(u)\\
    & = & \sum_{x\in A}\sum_{\chi\in \hat A} \chi(u-x/2)f(u-x)
  \end{eqnarray*}
  For each $x\in A$, the inner sum vanishes unless $x=2u$, in which case it is $\sqrt{|K|}$, showing that
  \begin{equation*}
    (\delta_K f)(u) = \sqrt{|K|}f(-u).
  \end{equation*}
  It follows that
  \begin{equation*}
    (\epsilon_\pm f)(u) = (f(u)\pm f(-u))/2,
  \end{equation*}
  are the projections onto the subspaces of even and odd functions.
\end{proof}
The order of an element $k\in K$ is an invariant of its $\Aut(K,e)$\dash orbit.
Therefore, unless $A$ has prime exponent, there are more than two $\Aut(K,e)$\dash orbits.
If $A$ has prime exponent, then $\Aut(K,e)$ acts transitively on the non-zero elements of $K$.
Consequently
\begin{cor}
  \label{cor:prime-exponent}
  The Weil representation of $\Aut(K,e)$ is never irreducible; the subspaces of $L^2(A)$ consisting of the even and odd functions are always invariant.
  These subspaces are irreducible if and only if $A$ has prime exponent.
\end{cor}
Theorem~\ref{theorem:decomposition} simplifies to
\begin{theorem}
  \label{theorem:decomposition-odd}
  Let $S$ be any subgroup of $\Aut(K,e)$.
  Then, with the $m_\pi$'s as in Theorem~\ref{theorem:decomposition}, $\sum m_\pi^2$ is the number of $S$\dash orbits in $K$.
\end{theorem}
In order to extend our results from subgroups of $\Aut(K,e)$ to subgroups of $B_0(G)$, we need
\begin{lemma}
  \label{lemma:conjugacy}
  The conjugacy class of an element $\sigma\in B_0(G)$ intersects $\Aut(K,e)$ if and only if the restriction of $\tilde\so$ to $K^{\st}$ is trivial.
\end{lemma}
\begin{proof}
  By direct calculation one sees that
  \begin{equation*}
    \tau\sigma\tau\inv(\W_k) = \tilde \so(k)\tilde\tau_1(\st(k)-k) \W_{\tau_2\st\tau_2\inv(k)}.
  \end{equation*}
  Therefore, the class of $\sigma$ intersects $\Aut(K,e)$ if and only if there exists $\tilde\tau_1\in \hat K$ such that $\tilde\so=\tilde\tau_1\circ(\st-1)$.
  But this happens if and only if $\tilde\so$ vanishes on the kernel of $\st-1$, which is nothing but $K^{\st}$.
\end{proof}
\begin{cor}
  \label{cor:abs-char}
  For $\sigma\in B_0(G)$, $\tr W(\sigma)=0$ if and only if the conjugacy class of $\sigma$ in $B_0(G)$ is disjoint from $\Aut(K,e)$.
\end{cor}
\begin{proof}
  By Theorem~\ref{theorem:abs-trace} $|\tr W(\sigma)|^2$ is the sum over the subgroup $K^{\st}$ of a character. 
  By Lemma~\ref{lemma:conjugacy}, this character is trivial on $K^{\st}$ if and only if the conjugacy class of $\sigma$ in $B_0(G)$ intersects $\Aut(K,e)$.
  The result now follows, since the sum of values taken by a multiplicative character on a finite group is zero if and only if the character is nontrivial.
\end{proof}
\begin{cor}
  \label{cor:decomposition}
  Let $S$ be any subgroup of $B_0(G)$ and the $m_\pi$'s be as in Theorem~\ref{theorem:decomposition}.
  Then $\sum_\pi m_\pi^2$ is the number of $S$\dash orbits in $K$ whose stabilizer in $S$ consists of elements which are conjugate in $B_0(G)$ to an element of $\Aut(K,e)$.
\end{cor}
\begin{proof}
  This result follows from Theorem~\ref{theorem:decomposition} and Lemma~\ref{lemma:conjugacy}.
\end{proof}
\section{An example}
\label{sec:example}
Take $A=(\Z/p^n\Z)^g$, where $p$ is an odd prime number and $n$ and $g$ are positive integers.
$\hat A$ can be identified with $A$ by defining the character $\chi_x$ corresponding to $x$ by
\begin{equation*}
  \chi_x(y) = e^{2\pi i x\cdot y/p^n},
\end{equation*}
where $x\cdot y$ denotes the usual dot product of $x$ and $y$.
Thus $K=(\Z/p^n\Z)^{2g}$, whose elements are denoted as pairs $(x,y)$, with $x,y\in A$.
The alternating form $e$ from Section~\ref{sec:autos} is given by
\begin{equation*}
  e((x,y),(x',y'))=\chi_{y'}(x)\chi_y(x')\inv=e^{2\pi i(x\cdot y'-x'\cdot y)/p^n}.
\end{equation*}
The group $\Aut(K,e)$ is the symplectic group, namely the group of automorphisms of $K$ which preserve the standard symplectic form $A\times A\to \Z/p^n\Z$ given by $((x,y),(x',y'))\mapsto x\cdot y'-x'\cdot y$.
It is not difficult to see that $\Aut(K,e)$ has $n+1$ orbits in $K$ with representatives $(p^re_1,0)$ for $r=0,\ldots,n$, where $e_1\in A$ denotes the first coordinate vector $(1,0,\ldots,0)$ in $A$.
The orbit of $(p^re_1,0)$ consists of vectors all of whose entries are divisible by $p^r$, but at least one entry is not divisible by $p^{r+1}$.

According to Theorem~\ref{theorem:odd-centre}, $\End_{\C[\Aut(K,e)]_c}$ has basis spanned by the operators
\begin{equation*}
  \Delta_r=\sum_{k\in p^rK}\W_k, \text{ for } r=0,\ldots,n.
\end{equation*}
\begin{prop}
  \label{prop:product}
  For $r,s\in \{0,\ldots,n\}$,
  \begin{equation*}
    \Delta_r\Delta_s=p^{2g(n-\max\{r,s\})}\Delta_{\max\{\min\{r,s\},n-\max\{r,s\}\}}.
  \end{equation*}
\end{prop}
\begin{proof}
  If $\Delta_r\Delta_s$ has the expansion $\sum_{x\in K}\phi_x \W_x$ then
  \begin{eqnarray}
    \nonumber
    \phi_x & = & \sum_{k\in p^rK,\;l\in p^sK,\;k+l=x}\tilde c(k,l)\\
    \label{eq:6}
    & = & \sum_{k\in p^rK\cap(x+p^sK)} \tilde c(k,x)\\
    \label{eq:7}
    & = & \sum_{l\in p^sK\cap(x+p^rK)} \tilde c(x,l).
  \end{eqnarray}
  In deducing the above identities we used the facts that $\tilde c$ is a bicharacter, and that $\tilde c(k,k)=\tilde c(l,l)= 1$ for all $k,l\in K$.
  Since $p^rK+p^sK=p^{\min\{r,s\}}K$, the above sums are empty (hence $0$) unless $x\in p^{\min\{r,s\}}K$.
  If $x\in p^sK$, then (\ref{eq:6}) evaluates $\phi_x$ as a sum over $p^{\max\{r,s\}}K$ of a character.
  This character is trivial if and only if $x\in p^{n-\max\{r,s\}}K$.
  When this happens the sum is $|p^{\max\{r,s\}}K|=p^{2g(n-\max\{r,s\})}$.
  In all other cases it is $0$.
  If $x\in p^rK$, then a similar argument using (\ref{eq:7}) works.
\end{proof}
We note, in particular, that $\End_{\C[\Aut(K,e)]_c}(L^2(A))$ is commutative.
As a consequence
\begin{theorem}
  \label{theorem:count}
  The Weil representation of $\Aut(K,e)$ has a multiplicity free decomposition into $n+1$ irreducible representations.
\end{theorem}
Suppose that $r\leq n-r$. For any $f\in L^2(A)$,
\begin{equation*}
  p^{-2gr}\Delta_{n-r} f (u) = p^{-gr}\sum_{x\in p^{n-r}rA} f(u-x)\sum_{y\in p^{n-r}r A}p^{-gr}\chi_y(u-x/2).
\end{equation*}
The inner sum vanishes unless $u-x/2\in p^rA$, in which case it is one.
Therefore
\begin{equation*}
  p^{-2gr}\Delta_{n-r} f = p^{-gr}\sum_{x\in p^{n-r}rA\cap (2u+p^rA)}f(u-x).
\end{equation*}
If $u\notin p^rA$ and $x\in p^{n-r}A$, since $x$ is also in $p^rA$, $x\notin 2u+p^rA$.
Therefore, the above sum is empty, hence zero.
On the other hand, if $u\in p^rA$, then $2u+p^rA=p^rA$, and the summation is over $p^{n-r}rA\cap p^rA=p^{n-r}A$.
We get an average over the $p^{n-r}A$\dash translates of $f$.
We have
\begin{equation*}
  (p^{-2gr}\Delta_{n-r}f) (u)=
  \begin{cases}
    \frac 1{|p^{n-r}A|}\sum_{x\in p^{n-r}A}f(u-x) & \text{ if } u\in p^rA,\\
    0 & \text{ otherwise}.
  \end{cases}
\end{equation*}
Thus, whenever $r\leq n-r$, $p^{-2gr}\Delta_{n-r}$ is the projection onto the subspace $V_{n-2r}$ of $L^2(A)$ consisting of functions that are supported on $p^rA$ and are invariant under $p^{n-r}A$, an invariant subspace for the Weil representation of dimension $p^{g(n-2r)}$.
Unless $n=2r$, this space can be further decomposed as $V_{n-2r}=V_{n-2r}^+\oplus V_{n-2r}^-$, consisting of the even and odd functions in $V_r$ of dimensions $(p^{g(n-2r)}\pm 1)/2$.
We have constructed two orthogonal nested sequences of invariant subspaces in $L^2(A)$:
\begin{equation*}
  V_n^\pm\supset V_{n-2}^\pm \supset \cdots \supset V_{n-2\lfloor n/2\rfloor}^\pm,
\end{equation*}
where $V^-_0=0$ if $n$ is even.
Let $U_{n-2r}^\pm$ denote the orthogonal complement of $V_{n-2r-2}^\pm$ in $V_{n-2r}^\pm$ for $r=0,\ldots,\lfloor n/2\rfloor -1$, and $U_{n-2\lfloor n/2\rfloor}^\pm = V_{n-2\lfloor n/2\rfloor}$.
Excluding $U_0^-$ when $n$ is even, these are $n+1$ pairwise orthogonal non-trivial invariant subspaces for  the Weil representation.
By of Theorem~\ref{theorem:count}, these must be the irreducible invariant subspaces for the Weil representation.
We get\pagebreak
\begin{theorem}
  \label{theorem:dimensions}
  The decomposition of $L^2(A)$ into irreducible invariant subspaces for the Weil representation is given by
  \begin{equation*}
    L^2(A)= \bigoplus_{r=0}^{\lfloor n/2\rfloor} (U_{n-2r}^+\oplus U_{n-2r}^-)
  \end{equation*}
  (with $U_0^-=0$ excluded when $n$ is even), where
  \begin{equation*}
    \dim U_{n-2r}^\pm = 
    \begin{cases} 
      (p^{g(n-2r)}-p^{g(n-2r-2)})/2 & \text{ if }0\leq r<\lfloor n/2\rfloor\\
      (\pm 1 + p^{g(n-r)})/2 & \text{ if } r=\lfloor n/2\rfloor.
    \end{cases}
  \end{equation*}
\end{theorem}
When $m\leq n$, there is a natural surjection $Sp_{2g}(\Z/p^n\Z)\to Sp_{2g}(\Z/p^m\Z)$.
Denote its kernel by $N_m$.
The level of a projective representation of $Sp_{2g}(\Z/p^n\Z)$ is defined as the largest $m$ for which $N_m$ acts by a scalar multiple of the identity.
Thus, for example, a projective representation has level one if it is not a scalar multiple of the identity and it comes from a projective representation of $Sp_{2g}(\Z/p\Z)$ by composition with the natural surjection $Sp_{2g}(\Z/p^n\Z)\to Sp_{2g}(\Z/p\Z)$.
\begin{lemma}
  \label{lemma:level}
  The level of the Weil representation of $Sp_{2g}(\Z/p^n\Z)$ is $n$.
\end{lemma}
\begin{proof}
  It suffices to show that there exists an element of $N_{n-1}$ which does not act by a scalar multiple of the identity, which is the same as showing that the absolute value of the character of the Weil representation evaluated at this element is less than $p^{gn}$.
  The matrix of an element of $N_{n-1}$ can be written as $g=I+p^{n-1}X$ for some $2n\times 2n$\dash matrix $X$.
  $K^g$ then consists of vectors in the kernel of $p^{n-1}X$.
  Thus if $X$ is non-zero modulo $p$, Theorem~\ref{theorem:absolute-char-odd} implies that $|\tr W(g)|<p^{gn}$, as desired.
\end{proof}
Recall that two projective $\rho_1$ and $\rho_2$ of a group on Hilbert spaces $\Hilb_1$ and $\Hilb_2$ are equivalent if there exists an isometry $W:\Hilb_1\to \Hilb_2$ such that $W\circ \rho_1(\sigma)$ and $\rho_2 \circ W(\sigma)$ differ by a constant for each element $\sigma$ of the group.
\begin{theorem}
  \label{theorem:levels}
  In the decomposition of Theorem~\ref{theorem:dimensions}, the Weil representation on $U_{n-2r}^\pm$ has level $n-2r$ for $r=0,\ldots \lfloor n/2\rfloor$.
  Let $q:SL_{2g}(\Z/p^n\Z)\to SL_{2g}(\Z/p^{n-2r}\Z)$ denote the natural quotient map.
  If $W_{n-2r}$ denotes the Weil representation of $Sp_{2g}(\Z/p^{n-2r}\Z)$ on $L^2((\Z/p^{n-2r}\Z)^g)$, then $W_{n-2r}\circ q$ is equivalent to the restriction of $W$ to $V_{n-2r}$.
\end{theorem}
\begin{proof}
  The first assertion follows from Lemma~\ref{lemma:level} and the second, which we now prove.
  Let $A'=\Z/p^{n-2r\Z}$.
  If $a\in p^rA$, write $a=p^r\tilde a$ for some $\tilde a\in A$.
  Let $\phi(a)$ be the image of $\tilde a$ under the natural surjection $A\to A'$.
  Then $\phi$ identifies $A'$ with $p^rA/p^{n-r}A$.
  For $f\in L^2(A')$ define $\tilde f\in L^2(A)$ by
  \begin{equation*}
    \tilde f(u)=
    \begin{cases}
      0 & \text{ if } u\notin p^{n-r}A,\\
      f(u') & \text{ if } u\in p^{n-r}A \text{ and } \phi(u')=u.
    \end{cases}
  \end{equation*}
  The image of the map $j:f\mapsto \tilde f$ is $V_{n-2r}$.
  For each $y'\in A'$ define $\chi_{y'}\in \widehat{A'}$ by $\chi_{y'}(x')=e^{2\pi i x'\cdot y'/p^{n-2r}}$.
  This identification is compatible with $\phi$ in the sense that $\chi_y(x)=\chi_{y'}(x')$ if $\phi(x)=x'$ and $\phi(y)=y'$.
  Then for $f\in L^2(A')$, $j(T_{x'}M_{\chi_{y'}}f)=T_xM_{\chi_y}j(f)$ if $\phi(x)=x'$ and $\phi(y)=y'$.
  Thus if $k'=(x',y')$ and $k=(x,y)$, we have that $j(\W_{k'}f)=\W_kj(f)$.
  Suppose $\sigma\in Sp_{2g}(\Z/p^nZ)$.
  Then $\phi(\sigma(x,y))=q(\sigma)(\phi(x),\phi(y))$ for all $x,y\in p^rA$.
  The defining property of $W(\sigma)$ is that
  \begin{equation*}
    W(\sigma)\W_kW(\sigma)\inv = \W_{\sigma(k)}.
  \end{equation*}
  Restricting to $V_{n-2r}=j(L^2(A))$ it follows from the above that
  \begin{equation*}
    W(\sigma)\W_{k'}W(\sigma)\inv = \W_{q(\sigma)(k')},
  \end{equation*}
  which is the defining property of $W_{n-2r}(q(\sigma))$.
\end{proof}
\begin{remark}
  When $n$ is even, the same results were obtained by Dipendra Prasad \cite[Theorem~2]{MR1478492}  using the lattice model of the Weil representation.
\end{remark}
\begin{remark}
  The decomposition of the Weil representation of\linebreak $Sp_{2g}(\Z/N\Z)$ for arbitrary $N$ can be deduced easily from the results of this section.
  Indeed, if $N=\prod_i p_i^{n_i}$ is the prime factorisation, then $Sp_{2g}(\Z/N\Z)=\prod_i Sp_{2g}(\Z/p_i^{n_i}\Z)$ and the Weil representation of $Sp_{2g}(\Z/N\Z)$ is the tensor product of the Weil representations of $Sp_{2g}(\Z/p_i^{n_i}\Z)$.
\end{remark}

\end{document}